\documentclass[a4paper]{article}

\usepackage{amsmath,amssymb,amsfonts,amsthm}
\usepackage[margin=2.54cm]{geometry}

\newtheorem{theorem}{Theorem}[section]
\newtheorem{lemma}[theorem]{Lemma}
\newtheorem{proposition}[theorem]{Proposition}

\newtheorem{hypothesis}[theorem]{Hypothesis}
\numberwithin{equation}{section}

\theoremstyle{remark}
\newtheorem{remark}[theorem]{Remark}
\newtheorem{example}[theorem]{Example}

\theoremstyle{definition}
\newtheorem{definition}[theorem]{Definition}

\newcommand{\Ric}{\mathop{\mathrm{Ric}}\nolimits}
\newcommand{\Ad}{\mathop{\mathrm{Ad}}\nolimits}

\newcommand{\tr}{\mathop{\mathrm{tr}}\nolimits}

\author{Artem Pulemotov\thanks{School of Mathematics and Physics, The
University of Queensland, St Lucia,~QLD 4072, Australia}~\thanks{Department of Mathematics, Cornell University, Ithaca, NY 14853, USA}~\thanks{Artem Pulemotov's research is supported under the Australian Research Council's Discovery Projects funding scheme (DP180102185).} \\
\small{\texttt{a.pulemotov@uq.edu.au}}}

\title{Maxima of curvature functionals and the prescribed Ricci curvature problem on homogeneous spaces}

\begin{document}

\maketitle

\begin{abstract}
Consider a compact Lie group $G$ and a closed Lie subgroup $H<G$. Let $\mathcal M$ be the set of $G$-invariant Riemannian metrics on the homogeneous space $M=G/H$. By studying variational properties of the scalar curvature functional on $\mathcal M$, we obtain an existence theorem for solutions to the prescribed Ricci curvature problem on~$M$. To illustrate the applicability of this result, we explore cases where $M$ is a generalised Wallach space and a generalised flag manifold.
\end{abstract}


\section{Introduction}\label{sec_intro}

The prescribed Ricci curvature problem is an important area of research in geometric analysis with close ties to flows and relativity. The first detailed results in this area were obtained in the early 1980s by D.~DeTurck. We invite the reader to see~\cite{RPLAMP15,APadd,ED17b} and references therein for some recent advances and a historical overview.

The present paper discusses the prescribed Ricci curvature problem in the framework of homogeneous spaces. More precisely, consider a compact connected Lie group $G$ and a closed connected Lie subgroup $H<G$. Denote by $M$ the homogeneous space~$G/H$. Suppose the dimension of $M$ is at least~3 and $\mathcal M$ is the set of $G$-invariant Riemannian metrics on~$M$. Let $S$ be the scalar curvature functional on~$\mathcal M$. The prescribed Ricci curvature problem for $G$-invariant metrics on $M$ consists in finding $g\in\mathcal M$ that satisfy the equation
\begin{align}\label{eq_PRC}
\Ric g=cT,
\end{align}
for some $c>0$, where $T$ is a given $G$-invariant $(0,2)$-tensor field. The study of~\eqref{eq_PRC} in the framework of homogeneous spaces was initiated in~\cite{AP16} and continued in~\cite{MGAP17}; see also~\cite{RH84,EDMH01,TB16}. It is on the basis of~\cite{AP16} that the first results about the Ricci iteration in the non-K\"ahler setting were obtained in~\cite{APYRsubm}. These results provided a new approach to uniformisation on homogeneous spaces. Moreover, they led to the discovery of several dynamical properties of the Ricci curvature.

Assume the $(0,2)$-tensor field $T$ lies in~$\mathcal M$. As~\cite[Lemma~2.1]{AP16} demonstrates, a metric $g\in\mathcal M$ satisfies~\eqref{eq_PRC} for some $c\in\mathbb R$ if and only if it is (up to scaling) a critical point of the functional $S$ on
\begin{align*}
\mathcal M_T=\{h\in\mathcal M\,|\,\tr_hT=1\}.
\end{align*}
This parallels the well-known variational interpretation of the Eisntein equation.
Indeed, a metric $g\in\mathcal M$ satisfies
\begin{align*}
\Ric g=\lambda g
\end{align*}
for some $\lambda\in\mathbb R$ if and only it is (up to scaling) a critical point of $S$ on
\begin{align*}
\mathcal M_1=\{h\in\mathcal M\,|\,M~\mbox{has volume 1 with respect to}~h\};
\end{align*}
see, e.g.,~\cite[\S1]{MWWZ86}. However, the restrictions of $S$ to $\mathcal M_T$ and $\mathcal M_1$ have substantially different properties. We will elaborate on this after we provide an overview of our main results.

According to~\cite[Theorem~1.1]{AP16}, if $H$ is a maximal connected Lie subgroup of~$G$, the functional $S$ attains its greatest value on~$\mathcal M_T$ at some $g\in\mathcal M_T$. It is easy to show that this $g$ satisfies~\eqref{eq_PRC} with $c>0$. In the present paper, we focus on the situation where the maximality assumption on $H$ does not hold. As~\cite[Proposition~3.1]{AP16} shows, in this situation, $S$ may fail to have a critical point on~$\mathcal M_T$. The work~\cite{MGAP17} provides a sufficient condition for $S$ to attain its greatest value on~$\mathcal M_T$. However, the usability of this result
is restricted by two main factors. First, the class of homogeneous spaces on which it applies, while broad, is far from exhaustive. All the examples known to date have isotropy representations that split into \emph{pairwise inequivalent} irreducible summands.
Second, the arguments in~\cite{MGAP17} impose rather strong requirements on~$T$.

In Section~\ref{sec_gen_results} of the present paper, we obtain a new sufficient condition for $S$ to attain its greatest value on~$\mathcal M_T$. We state this result as Theorem~\ref{thm_gen}. While it is similar in spirit to the sufficient condition given in~\cite{MGAP17}, it applies on a substantially larger class of homogeneous spaces. For instance, it can be used to analyse~\eqref{eq_PRC} on $Sp(n+1)/Sp(n)$ for $n\in\mathbb N$; see Remark~\ref{rem_spheres}. Topologically, this space is a $(4n+3)$-dimensional sphere. Its isotropy representation splits into four irreducible summands, three of which are equivalent to each other.

The requirements imposed on $T$ by Theorem~\ref{thm_gen} are different from the ones imposed by the sufficient condition of~\cite{MGAP17}. They appear to be substantially lighter in most situations. We illustrate this by considering two examples;
see Remarks~\ref{rem_compar1} and~\ref{rem_compar2} below.

Section~\ref{sec_Wallach} discusses the prescribed Ricci curvature problem on generalised Wallach spaces with inequivalent isotropy summands. A complete classification of such spaces is given in~\cite{ZCYKKL16,YN16}. They possess a number of interesting properties, and their geometry has been studied by several authors; see, e.g.,~\cite{AANS15,NAYN16,ZCYKKL16,AAYW17}. There are several infinite families and 10 isolated examples, excluding products, constructed out of exceptional Lie groups. In Section~\ref{sec_Wallach}, we show that Theorem~\ref{thm_gen} yields a sufficient condition for the existence of $g\in\mathcal M$ satisfying~\eqref{eq_PRC} for some $c>0$ in the case where $M$ is a generalised Wallach space with inequivalent isotropy summands. This condition is exceedingly easy to verify.

Section~\ref{sec_flag} is devoted to the prescribed Ricci curvature problem on generalised flag manifolds with up to five irreducible summands in the isotropy representation. Such manifolds form an important class of homogeneous spaces with applications across a variety of fields. For detailed discussions of their geometric properties, see~\cite[Chapter~7]{AA03} and the survey~\cite{AA15}. According to the classifications given in~\cite{AAIC10,SAIC11,AAICYS13}, there are numerous infinite families and isolated examples. In Section~\ref{sec_flag}, we describe a method for verifying the assumptions of Theorem~\ref{thm_gen} on generalised flag manifolds with up to five irreducible summands in the isotropy representation. For clarity, we provide a detailed analysis of the cases where $M$ equals $G_2/U(2)$ (with $U(2)$ corresponding to the long root of $G_2$) and $F_4/SU(3)\times SU(2)\times U(1)$.

The results of the present paper have no analogues in the theory of homogeneous Einstein metrics. Indeed, our main focus in on situations in which $S$ attains its greatest value on $\mathcal M_T$. However, as demonstrated by~\cite[Theorem~(2.4)]{MWWZ86} and~\cite[Theorem~1.2]{CB04}, the restriction of $S$ to the set $\mathcal M_1$ is typically unbounded above.

\section{Preliminaries}\label{sec_prelim}

As in Section~\ref{sec_intro}, consider a compact connected Lie group $G$ with Lie algebra $\mathfrak g$ and a closed connected subgroup $H<G$ with Lie algebra $\mathfrak h\subset\mathfrak g$. Let the homogeneous space $M=G/H$ be of dimension at least~3.
Choose a scalar product $Q$ on $\mathfrak g$ induced by a bi-invariant Riemannian metric on $G$. In what follows, $\oplus$ stands for the $Q$-orthogonal sum. Clearly,
\begin{align*}
\mathfrak g=\mathfrak m\oplus\mathfrak h
\end{align*}
for some $\Ad(H)$-invariant space~$\mathfrak m$. The representation $\Ad(H)|_{\mathfrak m}$ is equivalent to the isotropy representation of $G/H$. We standardly identify $\mathfrak m$ with the tangent space $T_HM$.

\subsection{The structure constants}\label{subsec_str_const}

Choose a $Q$-orthogonal $\Ad(H)$-invariant decomposition
\begin{align}\label{m_decomp}
\mathfrak m=\mathfrak m_1\oplus\cdots\oplus\mathfrak m_s
\end{align}
such that $\mathfrak m_i\ne\{0\}$ and $\Ad(H)|_{\mathfrak m_i}$ is irreducible for each $i=1,\ldots,s$. The space $\mathfrak m$ may admit more than one decomposition of this form. However, by Schur's lemma, the summands $\mathfrak m_1,\ldots,\mathfrak m_s$ are determined uniquely up to order if $\Ad(H)|_{\mathfrak m_i}$ and $\Ad(H)|_{\mathfrak m_j}$ are inequivalent whenever $i\ne j$. Let $d_i$ be the dimension of~$\mathfrak m_i$. It is easy to show that the number $s$ and the multiset $\{d_1,\ldots,d_s\}$ are independent of the chosen decomposition~\eqref{m_decomp}.

Denote by $B$ the Killing form of $\mathfrak g$. For every $i=1,\ldots,s$, because $\Ad(H)|_{\mathfrak m_i}$ is irreducible, there exists $b_i\ge0$ such that
\begin{align}\label{b_def}
B|_{\mathfrak m_i} = -b_iQ|_{\mathfrak m_i}.
\end{align}
The numbers $b_1,\ldots,b_s$ will help us write down a convenient formula for the scalar curvature of a metric on~$M$.

Given $\Ad(H)$-invariant subspaces $\mathfrak u_1$, $\mathfrak u_2$ and $\mathfrak u_3$ of $\mathfrak m$, define a tensor $\Delta(\mathfrak u_1,\mathfrak u_2,\mathfrak u_3)\in\mathfrak u_1\otimes\mathfrak u_2^*\otimes\mathfrak u_3^*$ by setting
\begin{align*}
\Delta(\mathfrak u_1,\mathfrak u_2,\mathfrak u_3)(X,Y)=\pi_{\mathfrak u_1}[X,Y],\qquad X\in\mathfrak u_2,~Y\in\mathfrak u_3,
\end{align*}
where $\pi_{\mathfrak u_1}$ stands for the $Q$-orthogonal projection onto~$\mathfrak u_1$. Let $\langle \mathfrak u_1\mathfrak u_2\mathfrak u_3\rangle$ be the squared norm of $\Delta(\mathfrak u_1,\mathfrak u_2,\mathfrak u_3)$ with respect to the scalar product on $\mathfrak u_1\otimes\mathfrak u_2^*\otimes\mathfrak u_3^*$ induced by $Q|_{\mathfrak u_1}$, $Q|_{\mathfrak u_2}$ and $Q|_{\mathfrak u_3}$. The fact that $Q$ comes from a bi-invariant metric on $G$ implies
\begin{align}\label{permut}
\langle\mathfrak u_1\mathfrak u_2\mathfrak u_3\rangle=\langle\mathfrak u_{\rho(1)}\mathfrak u_{\rho(2)}\mathfrak u_{\rho(3)}\rangle
\end{align}
for any permutation $\rho$ of the set $\{1,2,3\}$. Given $i,j,k\in\{1,\ldots,s\}$, denote
\begin{align*}
\langle\mathfrak m_i\mathfrak m_j\mathfrak m_k\rangle=[ijk].
\end{align*}
The numbers $([ijk])_{i,j,k=1}^s$ are called the \emph{structure constants} of the homogeneous space $M$; cf.~\cite[\S1]{MWWZ86}.
If $\mathcal J_{\mathfrak u_l}$ is a non-empty subset of $\{1,\ldots,s\}$ and
$\mathfrak u_l$ equals $\bigoplus_{i\in \mathcal J_{\mathfrak u_l}}\mathfrak m_i$
for each $l=1,2,3$, then
\begin{align}\label{<ijk>[ijk]}
\langle\mathfrak u_1\mathfrak u_2\mathfrak u_3\rangle=\sum_{i\in\mathcal J_{\mathfrak u_1}}\sum_{j\in\mathcal J_{\mathfrak u_2}}\sum_{k\in\mathcal J_{\mathfrak u_3}}[ijk].
\end{align}

\subsection{The scalar curvature functional and its extension}\label{subsec_sc_curv}

As in Section~\ref{sec_intro}, let $\mathcal M$ be the space of all $G$-invariant Riemannian metrics on~$M$. This space has a natural smooth manifold structure; see, e.g.,~\cite[pages~6318--6319]{YNERVS07} and~\cite[Subsection~4.1]{CB04}. The scalar curvature $S(g)$ of a metric $g\in\mathcal M$ is constant on~$M$. Therefore,
we may interpret $S(g)$ as the result of applying a functional $S:\mathcal M \to \mathbb R$ to~$g$. In what follows, we implicitly identify $g\in\mathcal M$ with the $\Ad(H)$-invariant scalar product induced by $g$ on $\mathfrak m$ via the identification of $T_HM$ and~$\mathfrak m$. To 
proceed, we need some notation. Namely, suppose $\mathfrak u$ and $\mathfrak v$ are subspaces of $\mathfrak m$ such that $\mathfrak v\subset\mathfrak u$. Let $\mathfrak u\ominus\mathfrak v$ be the $Q$-orthogonal complement of $\mathfrak v$ in~$\mathfrak u$. Consider bilinear forms $X$ and $Y$ on $\mathfrak u$ and $\mathfrak v$, respectively. Denote by $X|_\mathfrak v$ the restriction of $X$ to~$\mathfrak v$. If $X|_{\mathfrak v}$ is positive-definite, $\tr_XY$ stands for the trace of $Y$ with respect to~$X|_{\mathfrak v}$.

The scalar curvature of $g\in\mathcal M$ is given by the equality
\begin{align}\label{sc_curv_basf}
S(g)=-\frac12\tr_gB|_{\mathfrak m}-\frac14
|\Delta(\mathfrak m,\mathfrak m,\mathfrak m)|_g^2,
\end{align}
where $|\cdot|_g$ is the norm induced by~$g$ on $\mathfrak m\otimes\mathfrak m^*\otimes\mathfrak m^*$. If the decomposition~\eqref{m_decomp} is such that
\begin{align}\label{g_form}
g=\sum_{i=1}^sx_i\pi_{\mathfrak m_i}^*Q
\end{align}
for some $x_1,\ldots,x_s>0$, then
\begin{align}\label{sc_curv_formula}
S(g)&=\frac12\sum_{i=1}^s\frac{d_ib_i}{x_i}-\frac14\sum_{i,j,k=1}^s[ijk]\frac{x_k}{x_ix_j}.
\end{align}
For the derivation of formulas~\eqref{sc_curv_basf} and~\eqref{sc_curv_formula}, see, e.g.,~\cite[Chapter~7]{AB87} and~\cite[Lemma~3.2]{APYRsubm}. Given $g\in\mathcal M$, it is always possible to choose~\eqref{m_decomp} in such a way that~\eqref{g_form} holds. For the proof of this fact, see~\cite[page~180]{MWWZ86}.

Suppose $\mathfrak k$ is a Lie subalgebra of $\mathfrak g$ containing $\mathfrak h$ as a proper subset. It will be convenient for us to denote
\begin{align}\label{def_funny_subsp}
\mathfrak n=\mathfrak k\ominus\mathfrak h,\qquad \mathfrak l=\mathfrak g\ominus\mathfrak k.
\end{align}
Let $\mathcal M(\mathfrak k)$ be the space of $\Ad(H)$-invariant scalar products on $\mathfrak n$. In what follows, we assume $\mathcal M(\mathfrak k)$ is equipped with the topology inherited from the second tensor power of~$\mathfrak n^*$. Our further arguments require introducing an extension of the functional $S$ to~$\mathcal M(\mathfrak k)$. More precisely, define
\begin{align}\label{hat_S_def}
\hat S(h)=-\frac12\tr_hB|_{\mathfrak n}
-\frac12|\Delta(\mathfrak l,\mathfrak n,\mathfrak l)|_{\rm{mix}}^2
-\frac14
|\Delta(\mathfrak n,\mathfrak n,\mathfrak n)|_h^2,\qquad h\in\mathcal M(\mathfrak k),
\end{align}
where $|\cdot|_{\rm{mix}}$ is the norm induced by $Q|_{\mathfrak l}$ and $h$ on $\mathfrak l\otimes\mathfrak n^*\otimes\mathfrak l^*$ and $|\cdot|_h$ is the norm induced by~$h$ on $\mathfrak n\otimes\mathfrak n^*\otimes\mathfrak n^*$. If $\mathfrak k$ coincides with $\mathfrak g$, we identify $\mathcal M(\mathfrak k)$ with $\mathcal M$. In this case, the second term on the right-hand side of~\eqref{hat_S_def} vanishes, and $\hat S(h)$ equals~$S(h)$. If the decomposition~\eqref{m_decomp} is such that
\begin{align}\label{n_h_diag}
\mathfrak k=\Big(\bigoplus_{i\in J_{\mathfrak k}}\mathfrak m_i\Big)\oplus\mathfrak h,\qquad h=\sum_{i\in J_{\mathfrak k}}y_i\pi_{\mathfrak m_i}^*Q,\qquad y_i>0,
\end{align}
for some $J_{\mathfrak k}\subset\{1,\ldots,s\}$, then
\begin{align}\label{hat_S_formula}
\hat S(h)&=\frac12\sum_{i\in J_{\mathfrak k}}\frac{d_ib_i}{y_i}-\frac12\sum_{i\in J_{\mathfrak k}}\sum_{j,k\in J_{\mathfrak k}^c}\frac{[ijk]}{y_i}
-\frac14\sum_{i,j,k\in J_{\mathfrak k}}[ijk]\frac{y_k}{y_iy_j},
\end{align}
where $J_{\mathfrak k}^c$ is the complement of $J_{\mathfrak k}$ in $\{1,\ldots,s\}$; see~\cite[Lemma~2.19]{MGAP17}.

\subsection{Restrictions to hypersurfaces}

Throughout the rest of this paper, we fix $T\in\mathcal M$. As in Section~\ref{sec_intro}, define 
$\mathcal M_T$ as the set of those $g\in\mathcal M$ that satisfy the equality $\tr_gT=1$. Assume $\mathcal M_T$ has the smooth structure inherited from~$\mathcal M$. The following result is a special case of~\cite[Lemma~2.1]{AP16}. It provides a variational interpretation of the prescribed Ricci curvature equation~\eqref{eq_PRC} on homogeneous spaces. We will use it in the proof of Theorem~\ref{thm_gen} below.

\begin{proposition}\label{prop_var}
The Ricci curvature of a metric $g\in\mathcal M_T$ equals $cT$ for some $c\in\mathbb R$ if and only if $g$ is a critical point of the restriction of the scalar curvature functional $S$ to $\mathcal M_T$.
\end{proposition}

As in Subsection~\ref{subsec_sc_curv}, consider a Lie subalgebra $\mathfrak k$ of $\mathfrak g$ such that $\mathfrak h\subset\mathfrak k$ and $\mathfrak h\ne\mathfrak k$. Let $\mathfrak n$ and $\mathfrak l$ be given by~\eqref{def_funny_subsp}. Denote
\begin{align*}
\mathcal M_T(\mathfrak k)=\{h\in\mathcal M(\mathfrak k)\,|\,\tr_hT|_{\mathfrak n}=1\}.
\end{align*}
We assume $\mathcal M_T(\mathfrak k)$ is equipped with the topology it inherits from~$\mathcal M(\mathfrak k)$. Our next result introduces a important parameter associated with $\mathfrak k$.

\begin{proposition}\label{prop_sigma_finite}
The quantity $\sigma(\mathfrak k,T)$ defined by the formula
\begin{align*}
\sigma(\mathfrak k,T)=\sup\{\hat S(h)\,|\,h\in\mathcal M_T(\mathfrak k)\}
\end{align*}
satisfies
\begin{align*}
0\le\sigma(\mathfrak k,T)<\infty.
\end{align*}
\end{proposition}

\begin{proof}
The assertion follows immediately from the estimates on $\hat S$ obtained in~\cite[Lemmas~2.22 and~2.23]{MGAP17}. While these estimates are stated in~\cite{MGAP17} after Hypothesis~2.3 of that paper is imposed, their proofs do not require any assumptions other than those of Proposition~\ref{prop_sigma_finite}.
\end{proof}

\begin{remark}\label{rem_sigma_subalg}
If $\mathfrak k^+$ is a Lie subalgebra of $\mathfrak g$ containing $\mathfrak k$, then
\begin{align}\label{ineq_subs_sigma}
\sigma(\mathfrak k^+,T)\ge\sigma(\mathfrak k,T).
\end{align}
Indeed, fix $h\in\mathcal M_T(\mathfrak k)$ such that
\begin{align*}
\hat S(h)\in(\sigma(\mathfrak k,T)-\delta,\sigma(\mathfrak k,T)]
\end{align*}
for some $\delta>0$. It is possible to construct a curve $(h^+(t))_{t>\tr_QT}$ in the space $\mathcal M_T(\mathfrak k^+)$ satisfying the formula
\begin{align*}
\lim_{t\to\infty}\hat S(h^+(t))=\hat S(h);
\end{align*}
cf.~\cite[Proof of Lemma~2.30]{MGAP17}. The existence of such a curve implies
\begin{align*}
\sigma(\mathfrak k^+,T)\ge\hat S(h)>\sigma(\mathfrak k,T)-\delta
\end{align*}
Inequality~\eqref{ineq_subs_sigma} follows by letting $\delta$ go to~0.
\end{remark}

Next, we state a definition. It will help us formulate Theorem~\ref{thm_gen}.

\begin{definition}\label{def_T-apical}
We call $\mathfrak k$ a $T$-\emph{apical} subalgebra of $\mathfrak g$ if $\mathfrak k$ meets the following requirements:
\begin{enumerate}
\item
The inequality $\mathfrak k\ne\mathfrak g$ holds.

\item
There exists a scalar product $h\in\mathcal M_T(\mathfrak k)$ such that 
\begin{align}\label{max_attained}
\hat S(h)=\sigma(\mathfrak k,T).
\end{align}
\item
If $\mathfrak s$ is a maximal Lie subalgebra of $\mathfrak g$ containing $\mathfrak h$, then
\begin{align*}
\sigma(\mathfrak s,T)\le\sigma(\mathfrak k,T).
\end{align*}
\end{enumerate}
\end{definition}

We conclude this subsection with a result that provides a formula for $\sigma(\mathfrak k,T)$ when the representation $\Ad(H)|_{\mathfrak n}$ is irreducible. We will use this result in our study of generalised Wallach spaces and generalised flag manifolds below.

\begin{proposition}\label{prop_irred}
If $\Ad(H)|_{\mathfrak n}$ is irreducible, then $\mathcal M_T(\mathfrak k)$ consists of a single point. In this case, 
\begin{align}\label{sigma_irred}
\sigma(\mathfrak k,T)=-\frac{2\tr_QB|_{\mathfrak n}+\langle\mathfrak n\mathfrak n\mathfrak n\rangle+2\langle\mathfrak n\mathfrak l\mathfrak l\rangle}{4\tr_QT|_{\mathfrak n}}.
\end{align}
\end{proposition}

Before presenting the proof, let us restate~\eqref{sigma_irred} in terms of the structure constants of~$M$. This will help us with our computations in Sections~\ref{sec_Wallach} and~\ref{sec_flag}. If $\Ad(H)|_{\mathfrak n}$ is irreducible, it is possible to choose the decomposition~\eqref{m_decomp} so that $\mathfrak n=\mathfrak m_i$ for some $i=1,\ldots,s$. In this case,
\begin{align*}
T|_{\mathfrak n}=z_iQ|_{\mathfrak m_i},\qquad z_i>0.
\end{align*}
We use~\eqref{b_def} and~\eqref{<ijk>[ijk]} to find
\begin{align*}
\tr_QT|_{\mathfrak n}=d_iz_i,\qquad \tr_QB|_{\mathfrak n}=-d_ib_i,\qquad  \langle\mathfrak n\mathfrak n\mathfrak n\rangle
=[iii],\qquad \langle\mathfrak n\mathfrak l\mathfrak l\rangle
=\sum_{j,k\in\{1,\ldots,s\}\setminus\{i\}}[ijk].
\end{align*}
Thus, we can restate~\eqref{sigma_irred} as
\begin{align}\label{sigma_str_irred}
\sigma(\mathfrak k,T)=\frac1{d_iz_i}\bigg(\frac12d_ib_i-\frac14[iii]-\frac12\sum_{j,k\in\{1,\ldots,s\}\setminus\{i\}}[ijk]\bigg).
\end{align}

\begin{proof}[Proof of Proposition~\ref{prop_irred}.]
Assume $\Ad(H)|_{\mathfrak n}$ is irreducible. Then
\begin{align*}
\mathcal M(\mathfrak k)=\{yQ|_{\mathfrak n}\,|\,y>0\}.
\end{align*}
The trace $\tr_{yQ}T|_{\mathfrak n}$ equals $\frac1y\tr_QT|_{\mathfrak n}$ for all $y>0$. Consequently,
\begin{align*}
\mathcal M_T(\mathfrak k)=\{y_0Q|_{\mathfrak n}\},
\end{align*}
where $y_0=\tr_QT|_{\mathfrak n}$. Using~\eqref{hat_S_def} and~\eqref{permut}, we compute
\begin{align*}
\sigma(\mathfrak k,T)=\hat S(y_0Q|_{\mathfrak n})&=-\frac1{2y_0}\tr_QB|_{\mathfrak n}
-\frac1{2y_0}\langle\mathfrak l\mathfrak n\mathfrak l\rangle
-\frac1{4y_0}
\langle\mathfrak n\mathfrak n\mathfrak n\rangle \\ 
&=-\frac{2\tr_QB|_{\mathfrak n}+\langle\mathfrak n\mathfrak n\mathfrak n\rangle+2\langle\mathfrak n\mathfrak l\mathfrak l\rangle}{4\tr_QT|_{\mathfrak n}}.
\end{align*}
\end{proof}

\begin{remark}
Suppose $K$ is the connected Lie subgroup of $G$ whose Lie algebra equals~$\mathfrak k$. The irreducibility assumption on the representation $\Ad(H)|_{\mathfrak n}$ in Proposition~\ref{prop_irred} means that the homogeneous space $K/H$ is isotropy irreducible.
\end{remark}

\section{The general results}\label{sec_gen_results}

Our primary objective in this section is to state and prove an existence theorem for metrics satisfying~\eqref{eq_PRC} on the homogeneous space~$M$.

\subsection{Metrics with prescribed Ricci curvature and $T$-apical subalgebras}\label{subsec_ex_thm}

Theorems~\ref{thm_gen} 
below requires the following hypothesis. The class of homogeneous spaces for which this hypothesis holds is extensive. We discuss examples in Sections~\ref{sec_Wallach} and~\ref{sec_flag}.

\begin{hypothesis}\label{hyp_flag}
Every maximal 
Lie subalgebra $\mathfrak s$ of $\mathfrak g$ such that $\mathfrak h\subset\mathfrak s$ satisfies the following requirement:
if $\mathfrak u\subset\mathfrak s\ominus\mathfrak h$ and $\mathfrak v\subset\mathfrak g\ominus\mathfrak s$ are non-zero $\Ad(H)$-invariant spaces, then the representations $\Ad(H)|_{\mathfrak u}$ and $\Ad(H)|_{\mathfrak v}$ are inequivalent. 
\end{hypothesis}



\begin{remark}\label{rem_hyp1_ineq}
Hypothesis~\ref{hyp_flag} holds if the isotropy representation of $M$ splits into pairwise inequivalent irreducible summands; see~\cite[Proof of Proposition~4.1]{MGAP17}. However, it may be satisfied even if $M$ does not possess this property. We provide an example in Remark~\ref{rem_spheres} below.
\end{remark}

As above, consider a Lie subalgebra $\mathfrak k$ of $\mathfrak g$ such that $\mathfrak h\subset\mathfrak k$ and $\mathfrak h\ne\mathfrak k$. Let $\mathfrak n$ and $\mathfrak l$ be given by~\eqref{def_funny_subsp}.
We are now ready to state our theorem about the solvability of~\eqref{eq_PRC} on~$M$. In Sections~\ref{sec_Wallach} and~\ref{sec_flag}, we will use it to obtain existence results for metrics with prescribed Ricci curvature on generalised Wallach spaces and generalised flag manifolds.

\begin{theorem}\label{thm_gen}
Let Hypothesis~\ref{hyp_flag} hold. Suppose $\mathfrak k$ is a $T$-apical subalgebra of~$\mathfrak g$. If
\begin{align}\label{ineq_main_thm}
4\sigma(\mathfrak k,T)\tr_Q T|_{\mathfrak l}<-2\tr_QB|_{\mathfrak l}-\langle\mathfrak l\mathfrak l\mathfrak l\rangle,
\end{align}
then there exists $g\in\mathcal M_T$ such that $S(g)\ge S(g')$ for all~$g'\in\mathcal M_T$. The Ricci curvature of~$g$ equals $cT$ for some~$c>0$.
\end{theorem}

\begin{remark}\label{rem_spheres}
Suppose $G=Sp(k+1)$ and $H=Sp(k)$ for some $k\in\mathbb N$. In this case, $M$ is a sphere of dimension~$4k+3$. The isotropy representation of $M$ splits into four irreducible summands, three of which are equivalent to each other; see~\cite{WZ82}. Hypothesis~\ref{hyp_flag} holds, and Theorem~\ref{thm_gen} can be used to show that $S$ attains its global maximum on~$\mathcal M_T$ under certain conditions. The prescribed Ricci curvature problem on homogeneous spheres will be studied carefully in the forthcoming paper~\cite{TBAPYRWZ}.
\end{remark}

\begin{remark}
According to~\cite[Lemma~2.15]{MGAP17}, the quantity on the right-hand side of~\eqref{ineq_main_thm} is necessarily non-negative.
\end{remark}


\begin{remark}
As explained in~\cite[Remark~2.2]{MGAP17}, the restriction of $S$ to $\mathcal M_T$ is always bounded above but rarely bounded below.
\end{remark}

\begin{remark}
The properness of the restriction of $S$ to $\mathcal M_T$ is discussed in~\cite[Remark~2.11]{MGAP17}.
\end{remark}

We will prove Theorem~\ref{thm_gen} in Subsection~\ref{subsec_proofs}. In the meantime, let us restate condition~\eqref{ineq_main_thm} in terms of the structure constants of~$M$. Suppose the decomposition~\eqref{m_decomp} is such that the first formula in~\eqref{n_h_diag} holds
for some $J_{\mathfrak k}\subset\{1,\ldots,s\}$. In this case,
\begin{align}\label{tr_B_b}
\mathfrak l=\bigoplus_{i\in J_{\mathfrak k}^c}\mathfrak m_i,\qquad
\tr_QB|_{\mathfrak l}=-\sum_{i\in J_{\mathfrak k}^c}d_ib_i.
\end{align}
(Recall that $J_{\mathfrak k}^c$ denotes the complement of $J_{\mathfrak k}$.)
Exploiting~\eqref{<ijk>[ijk]}, we find that condition~\eqref{ineq_main_thm} holds if and only if
\begin{align}\label{cond_str_const}
\sigma(\mathfrak k,T)\tr_QT|_{\mathfrak l}<\frac12\sum_{i\in J_{\mathfrak k}^c}d_ib_i-\frac14\sum_{i,j,k\in J_{\mathfrak k}^c}[ijk].
\end{align}

The following strengthened version of Hypothesis~\ref{hyp_flag} is closely related to~\cite[Hypothesis~2.3]{MGAP17}. It will help us obtain our next result.

\begin{hypothesis}\label{hyp_strong}
Every Lie subalgebra $\mathfrak s$ of $\mathfrak g$ such that $\mathfrak h\subset\mathfrak s$ satisfies the following requirement:
if $\mathfrak u\subset\mathfrak s\ominus\mathfrak h$ and $\mathfrak v\subset\mathfrak g\ominus\mathfrak s$ are non-zero $\Ad(H)$-invariant spaces, then the representations $\Ad(H)|_{\mathfrak u}$ and $\Ad(H)|_{\mathfrak v}$ are inequivalent. 
\end{hypothesis}

\begin{remark}\label{rem_hyp2_ineq}
Hypothesis~\ref{hyp_strong} is necessarily satisfied if the isotropy representation of $M$ splits into pairwise inequivalent irreducible summands; see~\cite[Proof of Proposition~4.1]{MGAP17}.
\end{remark}

Theorem~\ref{thm_gen} is moot if $\mathfrak g$ does not have any $T$-apical subalgebras. Our next result shows that, under Hypothesis~\ref{hyp_strong}, at least one such subalgebra must exist. We prove this result in Subsection~\ref{subsec_pr_apic_ex}.

\begin{theorem}\label{thm_apic_ex}
Let Hypothesis~\ref{hyp_strong} hold. Suppose $\mathfrak h$ is not maximal in~$\mathfrak g$. Then $\mathfrak g$ has at least one $T$-apical subalgebra.
\end{theorem}

\subsection{Proof of Theorem~\ref{thm_gen}}\label{subsec_proofs}

Throughout this subsection, we assume Hypothesis~\ref{hyp_flag} holds. Our arguments will rely on the consequence of this hypothesis given by Lemma~\ref{lem_k=m} below. It provides a description of some of the Lie subalgebras of $\mathfrak g$ in terms of the decomposition~\eqref{m_decomp} chosen in Subsection~\ref{subsec_str_const}. We emphasise that, as explained in Remarks~\ref{rem_hyp1_ineq} and~\ref{rem_spheres}, the representations $\Ad(H)|_{\mathfrak m_i}$ and $\Ad(H)|_{\mathfrak m_j}$ may be equivalent for $i\ne j$.

\begin{lemma}\label{lem_k=m}
If $\mathfrak s$ is a maximal Lie subalgebra of $\mathfrak g$ such that $\mathfrak h\subset\mathfrak s$ and $\mathfrak h\ne\mathfrak s$, then there exists a unique set $J_{\mathfrak s}\subset\{1,\ldots,s\}$ satisfying the formula
\begin{align*}
\mathfrak s=\Big(\bigoplus_{i\in J_{\mathfrak s}}{\mathfrak m}_i\Big)\oplus\mathfrak h.
\end{align*}
\end{lemma}

\begin{proof}
It suffices to repeat the reasoning from~\cite[Proof of Lemma~2.12]{MGAP17} using Hypothesis~\ref{hyp_flag} instead of~\cite[Hypothesis~2.3]{MGAP17}.
\end{proof}

In this subsection, we assume $\mathfrak k$ is a $T$-apical subalgebra of~$\mathfrak g$. By definition,
\begin{align}\label{incl_k_proper}
\mathfrak h\subset\mathfrak k,\qquad\mathfrak h\ne\mathfrak k,\qquad \mathfrak k\ne\mathfrak g.
\end{align}
Let $\mathfrak k_1,\ldots,\mathfrak k_r$ be all the maximal Lie subalgebras of $\mathfrak g$ containing $\mathfrak h$ as a proper subset. Lemma~\ref{lem_k=m} implies that there are only finitely many such subalgebras. The fact that at least one exists follows from~\eqref{incl_k_proper}. Given $\tau>0$, define
\begin{align*}
\mathcal C(\tau,T)=\{g\in\mathcal M_T\,|\,g(X,X)\le\tau~\mbox{for all}~X\in\mathfrak m~\mbox{such that}~Q(X,X)=1\}.
\end{align*}
It is easy to verify that this set is compact in $\mathcal M_T$; cf.~\cite[Lemma~2.24]{MGAP17}. To prove Theorem~\ref{thm_gen}, we will need the following estimate for the scalar curvature functional~$S$.

\begin{lemma}\label{lem_eps+max}
Given $\epsilon>0$, there exists $\kappa(\epsilon)>0$ such that
\begin{align}\label{est_eps+max}
S(g)\le\epsilon+\max_{m=1,\ldots,r}\sigma(\mathfrak k_m,T)
\end{align}
for every $g\in\mathcal M_T\setminus\mathcal C(\kappa(\epsilon),T)$.
\end{lemma}

\begin{proof}
We repeat the arguments from~\cite[Proof of Lemma~2.28]{MGAP17} using Lemma~\ref{lem_k=m} above instead of~\cite[Lemma~2.12]{MGAP17}. As a result, we find
\begin{align*}
\hat S(g)\le\epsilon+\max_{m=1,\ldots,r}\sup\{\hat S(h)\,|\,h\in\mathcal M_T(\mathfrak k_m)\}, \qquad g\in\mathcal M_T\setminus \mathcal C(\kappa(\epsilon),T),
\end{align*}
for some $\kappa(\epsilon)>0$. With this estimate at hand,~\eqref{est_eps+max} follows from the definition of~$\sigma(\mathfrak k_m,T)$ and the fact that $\hat S$ coincides with $S$ on~$\mathcal M_T$.
\end{proof}

Next, we will demonstrate that $S$ attains its global maximum on $\mathcal M_T$ at some $g\in\mathcal M_T$ if condition~\eqref{ineq_main_thm} holds. Lemma~\ref{prop_var} will then imply equality~\eqref{eq_PRC} for this~$g$ with $c\in\mathbb R$. Once~\eqref{eq_PRC} is established, we will use Hypothesis~\ref{hyp_flag} to show that~$c>0$.

\begin{proof}[Proof of Theorem~\ref{thm_gen}.]
Since $\mathfrak k$ is a $T$-apical subalgebra of~$\mathfrak g$, there exists $h\in\mathcal M_T(\mathfrak k)$ such that~\eqref{max_attained} holds. Assume the decomposition~\eqref{m_decomp} is chosen so that $\mathfrak k$ and $h$ satisfy~\eqref{n_h_diag} for some $J_{\mathfrak k}\subset\{1,\ldots,s\}$. This assumption does not lead to loss of generality; see~\cite[page~180]{MWWZ86}. As above, let $\mathfrak l$ be given by~\eqref{def_funny_subsp}. For $t>\tr_QT|_{\mathfrak l}$, define a scalar product $h(t)\in\mathcal M_T$ by the formulas
\begin{align*}
h(t)=\sum_{j\in J_{\mathfrak k}}\phi(t)y_j\pi_{\mathfrak m_j}^*Q+\sum_{j\in J_{\mathfrak k}^c}t\pi_{\mathfrak m_j}^*Q,
\qquad
\phi(t)=\frac t{t-\tr_QT|_{\mathfrak l}}.
\end{align*}
Assuming~\eqref{ineq_main_thm} holds, we will show that 
\begin{align}\label{S>max}
S(h(t_0))>\max_{m=1,\ldots,r}\sigma(\mathfrak k_m,T)
\end{align}
for some~$t_0$. This inequality and Lemma~\ref{lem_eps+max} will imply the existence of $g\in\mathcal M_T$ such that $S(g)\ge S(g')$ for all $g'\in\mathcal M_T$.

With the aid of~\eqref{sc_curv_formula}, we find
\begin{align*}
S(h(t))&=\frac12\sum_{j\in J_{\mathfrak k}}\frac{d_jb_j}{\phi(t)y_j}+\frac12\sum_{j\in J_{\mathfrak k}^c}\frac{d_jb_j}t-\frac14\sum_{j,k,l\in J_{\mathfrak k}}[jkl]\frac{y_l}{\phi(t)y_jy_k}
-\frac14\sum_{j,k,l\in J_{\mathfrak k}^c}\frac{[jkl]}t
\\
&\hphantom{=}~-\frac12\sum_{j\in J_{\mathfrak k}}\sum_{k,l\in J_{\mathfrak k}^c}\frac{[jkl]}{\phi(t)y_j}-\frac14\sum_{j\in J_{\mathfrak k}}\sum_{k,l\in J_{\mathfrak k}^c}[jkl]\frac{\phi(t)y_j}{t^2}
\\
&\hphantom{=}~-\frac12\sum_{j,k\in J_{\mathfrak k}}\sum_{l\in J_{\mathfrak k}^c}[jkl]\frac{y_k}{ty_j}-\frac14\sum_{j,k\in J_{\mathfrak k}}\sum_{l\in J_{\mathfrak k}^c}[jkl]\frac t{\phi^2(t)y_jy_k}.
\end{align*}
The two terms in the last line are~0. Indeed, $[jkl]$ equals~0 if $j,k\in J_{\mathfrak k}$ and $l\in J_{\mathfrak k}^c$ because $\mathfrak k$ is a subalgebra of $\mathfrak g$; see~\cite[Lemma~2.14]{MGAP17}. Exploiting~\eqref{hat_S_formula}, \eqref{tr_B_b} and~\eqref{<ijk>[ijk]}, we conclude
\begin{align}\label{Sht_big_formula}
S(h(t))&=\frac1{\phi(t)}\bigg(\frac12\sum_{j\in J_{\mathfrak k}}\frac{d_jb_j}{y_j}-\frac12\sum_{j\in J_{\mathfrak k}}\sum_{k,l\in J_{\mathfrak k}^c}\frac{[jkl]}{y_j}-\frac14\sum_{j,k,l\in J_{\mathfrak k}}[jkl]\frac{y_l}{y_jy_k}\bigg) \notag
\\
&\hphantom{=}~+\frac1t\bigg(\frac12\sum_{j\in J_{\mathfrak k}^c}d_jb_j
-\frac14\sum_{j,k,l\in J_{\mathfrak k}^c}[jkl]\bigg)-\frac{\phi(t)}{4t^2}\sum_{j\in J_{\mathfrak k}}\sum_{k,l\in J_{\mathfrak k}^c}[jkl]y_j \notag
\\
&=\frac{\hat S(h)}{\phi(t)}-\frac1{4t}(2\tr_QB|_{\mathfrak l}
+\langle\mathfrak l\mathfrak l\mathfrak l\rangle)-\frac{\phi(t)}{4t^2}\sum_{j\in J_{\mathfrak k}}\sum_{k,l\in J_{\mathfrak k}^c}[jkl]y_j.
\end{align}
Since $\mathfrak k$ is $T$-apical,
\begin{align}\label{eq_lim_apic}
\lim_{t\to\infty}S(h(t))=\hat S(h)\ge\max_{m=1,\ldots,r}\sigma(\mathfrak k_m,T).
\end{align}
Our next step is to show that $\frac d{dt}S(h(t))<0$ when~\eqref{ineq_main_thm} holds and $t$ is large. This will imply the existence of $t_0$ satisfying~\eqref{S>max}.

It is clear that $\frac d{dt}S(h(t))$ has the same sign as $t^2\frac d{dt}S(h(t))$. Therefore, to prove that $\frac d{dt}S(h(t))<0$ for large~$t$, it suffices to show that
\begin{align*}
\lim_{t\to\infty}t^2\frac d{dt}S(h(t))<0.
\end{align*}
Using~\eqref{Sht_big_formula}, we obtain
\begin{align*}
t^2\frac d{dt}S(h(t))&=\hat S(h)\tr_QT|_{\mathfrak l}+\frac12\tr_QB|_{\mathfrak l}+\frac14\langle\mathfrak l\mathfrak l\mathfrak l\rangle+\frac{2t-\tr_QT|_{\mathfrak l}}{4(t-\tr_QT|_{\mathfrak l})^2}\sum_{j\in J_{\mathfrak k}}\sum_{k,l\in J_{\mathfrak k}^c}[jkl]y_j.
\end{align*}
If~\eqref{ineq_main_thm} holds, then
\begin{align*}
\lim_{t\to\infty}t^2\frac d{dt}S(h(t))&=\hat S(h)\tr_QT|_{\mathfrak l}+\frac12\tr_QB|_{\mathfrak l}+\frac14\langle\mathfrak l\mathfrak l\mathfrak l\rangle
\\
&=\sigma(\mathfrak k,T)\tr_QT|_{\mathfrak l}+\frac{2\tr_QB|_{\mathfrak l}+\langle\mathfrak l\mathfrak l\mathfrak l\rangle}4<0.
\end{align*}
In this case, $\frac d{dt}S(h(t))<0$ for large~$t$, which implies the existence of $t_0$ such that
\begin{align*}
S(h(t_0))>\lim_{t\to\infty}S(h(t)).
\end{align*}
Formula~\eqref{eq_lim_apic} plainly shows that $t_0$ satisfies~\eqref{S>max}. Using Lemma~\ref{lem_eps+max} with
\begin{align*}
\epsilon=\frac{S(h(t_0))-\max_{m=1,\ldots,r}\sigma(\mathfrak k_m,T)}2>0,
\end{align*}
we conclude that
\begin{align}\label{est_outside_C}
S(g')<S(h(t_0)),\qquad g'\in\mathcal M_T\setminus\mathcal C(\kappa(\epsilon),T).
\end{align}
Since the set~$\mathcal C(\kappa(\epsilon),T)$ is compact, the functional $S$ attains its global maximum on $\mathcal C(\kappa(\epsilon),T)$ at some $g\in\mathcal C(\kappa(\epsilon),T)$. Inequality~\eqref{est_outside_C} implies that, in fact, $S(g')\le S(g)$ for all $g'\in\mathcal M_T$. By Proposition~\ref{prop_var}, $g$ satisfies~\eqref{eq_PRC} for some $c\in\mathbb R$. The proof will be complete if we show that $c>0$.

According to Bochner's theorem (see~\cite[Theorem~1.84]{AB87}), there are no $G$-invariant Riemannian
metrics on $M$ with negative-definite Ricci curvature. It follows that $c\ge0$. Let us show that no $G$-invariant metric on $M$ can be Ricci-flat. This will imply~$c\ne0$.

We proceed by contradiction. Assume $M$ supports a Ricci-flat $G$-invariant metric. Using Bochner's theorem again, we can show that
\begin{align*}
[\mathfrak m,\mathfrak h]=\{0\};
\end{align*}
cf.~\cite[Subsection~2.7]{MGAP17}. Consequently, given~$i=1,\ldots,s$, the representation $\Ad(H)|_{\mathfrak m_i}$ is trivial. By irreducibility, the dimension of this representation is~1.

Let us prove that the maximal Lie subalgebra $\mathfrak k_1$ of $\mathfrak g$ fails to satisfy the requirement of Hypothesis~\ref{hyp_flag}. This will give us the contradiction we are seeking. By Lemma~\ref{lem_k=m}, the space $\mathfrak k_1\ominus\mathfrak h$ equals $\bigoplus_{i\in J_{\mathfrak k_1}}\mathfrak m_i$ for some $J_{\mathfrak k_1}\subset\{1,\ldots,s\}$. If $k$ and $l$ lie in $J_{\mathfrak k_1}$ and $J_{\mathfrak k_1}^c$, respectively, then
\begin{align*}
\mathfrak m_k\subset\mathfrak k_1\ominus\mathfrak h,\qquad \mathfrak m_l\subset\mathfrak g\ominus\mathfrak k_1.
\end{align*}
As we showed above, the representations $\Ad(H)|_{\mathfrak m_k}$ and $\Ad(H)|_{\mathfrak m_l}$ are trivial and 1-dimensional.
Therefore, these representations are equivalent, which means $\mathfrak k_1$ does not satisfy the requirement of Hypothesis~\ref{hyp_flag}.
\end{proof}

\subsection{Proof of Theorem~\ref{thm_apic_ex}}\label{subsec_pr_apic_ex}

In this subsection, we assume Hypothesis~\ref{hyp_strong} holds and $\mathfrak h$ is not maximal in~$\mathfrak g$. The proof of Theorem~\ref{thm_apic_ex} will rely on the following result.

\begin{lemma}\label{lem_key_apic_ex}
If $\mathfrak s$ is a Lie subalgebra of $\mathfrak g$ such that $\mathfrak h\subset\mathfrak s$ and $\mathfrak h\ne\mathfrak s$, then there exist a Lie subalgebra $\mathfrak k\subset\mathfrak s$ and a scalar product $h\in\mathcal M_T(\mathfrak k)$ satisfying the formulas 
\begin{align}\label{apic_ex_flas}
\mathfrak h\subset\mathfrak k,\qquad \mathfrak h\ne\mathfrak k,\qquad \hat S(h)=\sigma(\mathfrak k,T)\ge\sigma(\mathfrak s,T).
\end{align}
\end{lemma}

\begin{proof}
Denote by $p$ the dimension of $\mathfrak s\ominus\mathfrak h$. We proceed by induction in~$p$. The space $\mathcal M_T(\mathfrak s)$ consists of a single point $h_0$ if $p=1$. Formulas~\eqref{apic_ex_flas} are evident in this case for $\mathfrak k=\mathfrak s$ and $h=h_0$.

Let the assertion of the lemma hold when $p\le m$. Our goal is to prove it supposing $p=m+1$. If $\hat S$ has a global maximum on $\mathcal M_T(\mathfrak s)$ at some $h_0\in\mathcal M_T(\mathfrak s)$, then formulas~\eqref{apic_ex_flas} are easy to check for~$\mathfrak k=\mathfrak s$ and $h=h_0$. Therefore, we will focus on the case where the supremum of $\hat S$ over $\mathcal M_T(\mathfrak s)$ is not attained. In this case, $\mathfrak k$ and $\mathfrak s$ cannot coincide.

Suppose $\mathfrak s_1,\ldots,\mathfrak s_q$ are all the maximal Lie subalgebras of $\mathfrak s$ containing $\mathfrak h$ as a proper subset. By~\cite[Corollary~2.13 and Remark~2.34]{MGAP17}, there are only finitely many such subalgebras. Under our assumptions, at least one must exist. Otherwise, $\mathfrak h$ would be maximal in~$\mathfrak s$, and~\cite[Lemma~2.32 and Remark~2.34]{MGAP17} would imply that $\hat S$ has a global maximum on~$\mathcal M_T(\mathfrak s)$.

Choose an index $i_0$ such that
\begin{align*}
\sigma(\mathfrak s_{i_0},T)=\max_{i=1,\ldots,q}\sigma(\mathfrak s_i,T).
\end{align*}
Evidently, the dimension of $\mathfrak s_{i_0}$ is less than or equal to~$m$.
By the induction hypothesis, there exist a Lie subalgebra $\mathfrak k\subset\mathfrak s_{i_0}$ and a scalar product $h\in\mathcal M_T(\mathfrak k)$ satisfying formulas~\eqref{apic_ex_flas} with $\mathfrak s$ replaced by~$\mathfrak s_{i_0}$. The proof will be complete if we show that
\begin{align}\label{sigma<sigma}
\sigma(\mathfrak s,T)\le\sigma(\mathfrak s_{i_0},T).
\end{align}

Pick $\epsilon>0$. According to~\cite[Lemma~2.28 and Remark~2.34]{MGAP17},
\begin{align}\label{h'<eps+sigma}
\hat S(h')\le\epsilon+\sigma(\mathfrak s_{i_0},T)
\end{align}
whenever $h'\in\mathcal M_T(\mathfrak s)$ lies outside some compact subset $\mathcal C_\epsilon$ of $\mathcal M_T(\mathfrak s)$. We will demonstrate that, in fact,~\eqref{h'<eps+sigma} holds for all $h'\in\mathcal M_T(\mathfrak s)$. Estimate~\eqref{sigma<sigma} will follow by taking the supremum on the left-hand side of~\eqref{h'<eps+sigma} and letting $\epsilon$ go to~0.

Since $\mathcal C_{\epsilon}$ is compact, there exists $h_0'\in\mathcal C_\epsilon$ such that 
\begin{align*}
\hat S(h_0')\ge\hat S(h'),\qquad h'\in\mathcal C_{\epsilon}.\end{align*}
Clearly, if
\begin{align*}
\hat S(h_0')>\epsilon+\sigma(\mathfrak s_{i_0},T),
\end{align*}
then $\hat S$ has a global maximum on $\mathcal M_T(\mathfrak s)$ at~$h_0'$. However, we are assuming that the supremum of $\hat S$ over $\mathcal M_T(\mathfrak s)$ is not attained. Thus,~\eqref{h'<eps+sigma} holds for all $h'\in\mathcal M_T(\mathfrak s)$.
\end{proof}

\begin{remark}
In view of Remark~\ref{rem_sigma_subalg}, we can replace the third formula in~\eqref{apic_ex_flas} by
\begin{align*}
\hat S(h)=\sigma(\mathfrak k,T)=\sigma(\mathfrak s,T).
\end{align*}
\end{remark}


As above, denote by $\mathfrak k_1,\ldots,\mathfrak k_r$ the maximal Lie subalgebras of $\mathfrak g$ containing $\mathfrak h$ as a proper subset. With Hypothesis~\ref{hyp_strong} at hand, one can invoke Lemma~\ref{lem_k=m} or~\cite[Corollary~2.13 and Remark~2.34]{MGAP17} to conclude that there are only finitely many such subalgebras. The fact that at least one exists follows from the assumption that $\mathfrak h$ is not maximal in~$\mathfrak g$.

\begin{proof}[Proof of Theorem~\ref{thm_apic_ex}.]
Choose $j_0$ such that
\begin{align*}
\sigma(\mathfrak k_{j_0},T)=\max_{j=1,\ldots,r}\sigma(\mathfrak k_j,T).
\end{align*}
Applying Lemma~\ref{lem_key_apic_ex} with $\mathfrak s=\mathfrak k_{j_0}$, we obtain a Lie subalgebra $\mathfrak k$ of $\mathfrak k_{j_0}$ and a scalar product $h\in\mathcal M_T(\mathfrak k)$ satisfying formulas~\eqref{apic_ex_flas} with $\mathfrak s=\mathfrak k_{j_0}$. It is obvious that $\mathfrak k\ne\mathfrak g$ and~\eqref{max_attained} holds. Furthermore,
\begin{align*}
\sigma(\mathfrak k_j,T)\le\sigma(\mathfrak k_{j_0},T)\le\sigma(\mathfrak k,T),\qquad j=1,\ldots,r.
\end{align*}
Thus, $\mathfrak k$ is $T$-apical.
\end{proof}

\section{Generalised Wallach spaces}\label{sec_Wallach}

Throughout Section~\ref{sec_Wallach}, we assume $M$ is a generalised Wallach space, as defined in~\cite{YN16}. Theorem~\ref{thm_gen} and Proposition~\ref{prop_irred} yield an easy-to-verify sufficient condition for the solvability of~\eqref{eq_PRC} on such spaces. Our goal is to state this condition, prove it and consider an example.

Since $M$ is a generalised Wallach space, the group $G$ is semisimple. Consequently, the Killing form $B$ is non-degenerate. It will be convenient for us to set $Q=-B$.  The definition of a generalised Wallach space implies that $s=3$ in the decomposition~\eqref{m_decomp}. Thus, the formula
\begin{align}\label{m3_dec}
\mathfrak m=\mathfrak m_1\oplus\mathfrak m_2\oplus\mathfrak m_3
\end{align}
holds. 

We assume the representations $\Ad(H)|_{\mathfrak m_1}$, $\Ad(H)|_{\mathfrak m_2}$ and $\Ad(H)|_{\mathfrak m_3}$ are pairwise inequivalent. 
The summands $\mathfrak m_1,\mathfrak m_2,\mathfrak m_3$ in the decomposition~\eqref{m3_dec} are determined uniquely up to order. Appealing to the definition of a generalised Wallach space again, we obtain the inclusion
\begin{align}\label{incl_Wallach_mi}
[\mathfrak m_i,\mathfrak m_i]\subset\mathfrak h,\qquad i=1,2,3.
\end{align}
As a consequence, the structure constant $[ijk]$ vanishes unless $(i,j,k)$ is a permutation of $\{1,2,3\}$. We assume $[123]$ is positive. If $[123]=0$, then all the metrics in $\mathcal M$ have the same Ricci curvature; cf.~\cite[Lemma~3.2]{APYRsubm}. In this case, the analysis of~\eqref{eq_PRC} is easy.

Since $\Ad(H)|_{\mathfrak m_1}$, $\Ad(H)|_{\mathfrak m_2}$ and $\Ad(H)|_{\mathfrak m_3}$ are pairwise inequivalent, the formula
\begin{align*}
T=-z_1\pi_{\mathfrak m_1}^*B-z_2\pi_{\mathfrak m_2}^*B-z_3\pi_{\mathfrak m_3}^*B
\end{align*}
holds for some $z_1,z_2,z_3>0$. Denote by $d$ the dimension of~$M$. Clearly, $d=d_1+d_2+d_3$.
We are now ready to state the main result of this section.

\begin{theorem}\label{thm_Wallach}
Suppose $p\in\{1,2,3\}$ is such that
\begin{align}\label{cond_Wallach}
\frac{d_p-2[123]}{2d_pz_p}=\max_{i=1,2,3}\frac{d_i-2[123]}{2d_iz_i}.
\end{align}
If
\begin{align}\label{cond2_W}
(d_p-2[123])\sum_{i\in\{1,2,3\}\setminus\{p\}}d_iz_i<(d-d_p)d_pz_p,
\end{align}
then there exists a metric $g\in\mathcal M_T$ with Ricci curvature $cT$ for some~$c>0$.
\end{theorem}

\begin{remark}\label{rem_W_const}
The numbers $d_1$, $d_2$, $d_3$ and $[123]$ for the generalised Wallach spaces satisfying the assumptions of this section are given in~\cite[Table~1]{YN16}; see also~\cite[Tables~1 and~2]{ZCYKKL16}.
\end{remark}


To prove the theorem, we need the following result.

\begin{lemma}\label{lem_Wallach_subal}
The proper Lie subalgebras of $\mathfrak g$ containing $\mathfrak h$ as a proper subset are
\begin{align*}
\mathfrak k_1=\mathfrak m_1\oplus\mathfrak h,\qquad \mathfrak k_2=\mathfrak m_2\oplus\mathfrak h,\qquad \mathfrak k_3=\mathfrak m_3\oplus\mathfrak h.
\end{align*}
\end{lemma}

\begin{proof}
Formula~\eqref{incl_Wallach_mi} implies that 
\begin{align*}
[\mathfrak k_i,\mathfrak k_i]\subset\mathfrak k_i,\qquad i=1,2,3.
\end{align*}
Thus, $\mathfrak k_i$ is indeed a Lie subalgebra of $\mathfrak g$ for $i=1,2,3$.

As we noted above, the summands in~\eqref{m3_dec} are determined uniquely up to order. Using this observation, one can show that every $\Ad(H)$-invariant subspace of $\mathfrak m$ appears as $\bigoplus_{i\in\mathcal J}\mathfrak m_i$, where $\mathcal J\subset\{1,2,3\}$. Let $\mathfrak k$ be a Lie subalgebra of $\mathfrak g$ such that $\mathfrak h\subset\mathfrak k$, $\mathfrak h\ne\mathfrak k$ and $\mathfrak k\ne\mathfrak g$. We will prove that $\mathfrak k=\mathfrak k_i$ for some $i\in\{1,2,3\}$. Clearly, $\mathfrak k\ominus\mathfrak h$ is a proper non-zero $\Ad(H)$-invariant subspace of~$\mathfrak m$. Therefore, $\mathfrak k$ equals $\mathfrak m_i\oplus\mathfrak h$ for some $i\in\{1,2,3\}$ or $\mathfrak m_j\oplus\mathfrak m_k\oplus\mathfrak h$ for some distinct $j,k\in\{1,2,3\}$. Since $[123]>0$, the space $\mathfrak m_j\oplus\mathfrak m_k\oplus\mathfrak h$ cannot be a Lie subalgebra of~$\mathfrak g$. This means $\mathfrak k=\mathfrak k_i$ for some $i\in\{1,2,3\}$.
\end{proof}

\begin{proof}[Proof of Theorem~\ref{thm_Wallach}.]
Because $\Ad(H)|_{\mathfrak m_1}$, $\Ad(H)|_{\mathfrak m_2}$ and $\Ad(H)|_{\mathfrak m_3}$ are pairwise inequivalent, Hypothesis~\ref{hyp_flag} holds for~$M$; see Remark~\ref{rem_hyp1_ineq}. Let us show that 
\begin{align*}
\mathfrak k_p=\mathfrak m_p\oplus\mathfrak h
\end{align*}
is a $T$-apical subalgebra of~$\mathfrak g$. This will enable us to apply Theorem~\ref{thm_gen}. The existence of $g$ under condition~\eqref{cond2_W} will follow immediately.

It is obvious that $\mathfrak k_p$ cannot equal~$\mathfrak g$. Ergo, $\mathfrak k_p$ satisfies requirement~1 of Definition~\ref{def_T-apical}. By Proposition~\ref{prop_irred}, the space $\mathcal M_T(\mathfrak k_p)$ consists of a single point. Denoting this point by $h_0$, we see that equality~\eqref{max_attained} holds with $h=h_0$ and $\mathfrak k=\mathfrak k_p$. Consequently, $\mathfrak k_p$ satisfies requirement~2 of Definition~\ref{def_T-apical}. By Lemma~\ref{lem_Wallach_subal}, the maximal subalgebras of $\mathfrak g$ containing $\mathfrak h$ as a proper subset are $\mathfrak k_1$, $\mathfrak k_2$ and $\mathfrak k_3$. Since $Q=-B$, the numbers $b_1$, $b_2$ and $b_3$ given by~\eqref{b_def} all equal~1. Exploiting~\eqref{sigma_str_irred}, \eqref{incl_Wallach_mi} and~\eqref{cond_Wallach}, we obtain
\begin{align*}
\sigma(\mathfrak k_i,T)=\frac{d_i-2[123]}{2d_iz_i}\le\frac{d_p-2[123]}{2d_pz_p}=\sigma(\mathfrak k_p,T),\qquad i=1,2,3.
\end{align*}
This means $\mathfrak k_p$ satisfies requirement~3 of Definition~\ref{def_T-apical}.

By Theorem~\ref{thm_gen} and formula~\eqref{cond_str_const}, a metric $g\in\mathcal M_T$ with Ricci curvature $cT$ for some $c>0$ exists if
\begin{align*}
\sigma(\mathfrak k_p,T)\tr_QT|_{\mathfrak g\ominus\mathfrak k_p}<\frac12\sum_{i\in\{1,2,3\}\setminus\{p\}}d_i-\frac14\sum_{i,j,k \in\{1,2,3\}\setminus\{p\}}[ijk].
\end{align*}
Using~\eqref{sigma_str_irred} and~\eqref{incl_Wallach_mi} again, we find
\begin{align*}
\sigma(\mathfrak k_p,T)\tr_QT|_{\mathfrak g\ominus\mathfrak k_p}&=\frac{d_p-2[123]}{2d_pz_p}\sum_{i\in\{1,2,3\}\setminus\{p\}}d_iz_i,\\
\frac12\sum_{i\in\{1,2,3\}\setminus\{p\}}d_i&-\frac14\sum_{i,j,k \in\{1,2,3\}\setminus\{p\}}[ijk]=\frac{d-d_p}2.
\end{align*}
The assertion of Theorem~\ref{thm_Wallach} follows immediately.
\end{proof}

\begin{example}
Suppose $M$ is the generalised Wallach space $E_6/Sp(3)\times Sp(1)$. Let the decomposition~\eqref{m3_dec} be as in~\cite{YN16}. We have
\begin{align*}
d_1=14,\qquad d_2=28,\qquad d_3=12,\qquad [123]=7/2;
\end{align*}
see Remark~\ref{rem_W_const} above. A simple computation shows that
\begin{align*}
\frac{d_1-2[123]}{2d_1z_1}=\frac1{4z_1},\qquad 
\frac{d_2-2[123]}{2d_2z_2}=\frac3{8z_2},\qquad 
\frac{d_3-2[123]}{2d_3z_3}=\frac5{24z_3}.
\end{align*}
If
\begin{align}\label{frac1}
z_1/z_2\le2/3,\qquad z_1/z_3\le6/5,
\end{align}
then~\eqref{cond_Wallach} holds for $p=1$. In this case, according to Theorem~\ref{thm_Wallach}, a Riemannian metric with Ricci curvature $cT$ exists for some $c>0$ provided that
\begin{align*}
7z_2+3z_3<20z_1.
\end{align*}
If 
\begin{align}\label{frac2}
z_1/z_2\ge2/3,\qquad z_2/z_3\le9/5,
\end{align}
then~\eqref{cond_Wallach} holds for $p=2$. In this case, a metric with Ricci curvature $cT$ exists for some $c>0$ as long as
\begin{align*}
21z_1+18z_3<52z_2.
\end{align*}
Finally, if
\begin{align}\label{frac3}
z_1/z_3\ge6/5,\qquad z_2/z_3\ge9/5,
\end{align}
then~\eqref{cond_Wallach} is satisfied for $p=3$. In this situation, a metric with Ricci curvature $cT$ exists for some $c>0$ as long as
\begin{align*}
5z_1+10z_2<36z_3.
\end{align*}
It is easy to see that~\eqref{frac1}, \eqref{frac2} or~\eqref{frac3} necessarily holds for the triple $(z_1,z_2,z_3)$.
\end{example}

\section{Generalised flag manifolds}\label{sec_flag}

Suppose $M$ is a generalised flag manifold, as defined in~\cite{SAIC11}. Let the number $s$ in the decomposition~\eqref{m_decomp} be less than or equal to~5. The results of Sections~\ref{sec_prelim} and~\ref{sec_gen_results} yield an easy-to-verify sufficient condition for the solvability of~\eqref{eq_PRC} on~$M$. This condition involves inequalities depending on~$s$ and on whether $M$ is of type~I, II, A or~B in the terminology of~\cite{SAIC11,AAIC10,AAICYS13}. One could state it in the form of a theorem as we did in Section~\ref{sec_Wallach}; however,
such a theorem would be exceedingly bulky. Instead, we will illustrate the application of the results of Sections~\ref{sec_prelim} and~\ref{sec_gen_results} by considering two examples of~$M$. Only minor changes are required to make our reasoning work for other~$M$.

The definition of a generalised flag manifold implies that the group $G$ is semisimple. Therefore, the Killing form $B$ is non-degenerate. It will be convenient for us set $Q=-B$ throughout Section~\ref{sec_flag}. This means $b_i=1$ in~\eqref{b_def} for all $i=1,\ldots,s$. Because $M$ is a generalised flag manifold, the representations $\Ad(H)|_{\mathfrak m_i}$ and $\Ad(H)|_{\mathfrak m_j}$ are inequivalent when $i\ne j$; see~\cite[page~101]{AA03}. As a consequence, Hypotheses~\ref{hyp_flag} and~\ref{hyp_strong} are satisfied.
To determine the dimensions $d_i$ and the structure constants for specific choices of~$M$, refer to~\cite[Chapter~7]{AB87} and~\cite{AAIC10,SAIC11,AAICYS13}.

We are now ready to illustrate the application of the results of Sections~\ref{sec_prelim} and~\ref{sec_gen_results}.

\begin{example}\label{example_G2U2}
Suppose $M$ is the generalised flag manifold $G_2/U(2)$ in which $U(2)$ corresponds to the long root of~$G_2$. Let the decomposition~\eqref{m_decomp} be as in~\cite{SAIC11}. Then
\begin{align*}
s=3,\qquad d_1=d_3=4,\qquad d_2=2,\qquad [123]=1/2,\qquad [112]=2/3.
\end{align*}
The structure constant $[ijk]$ vanishes if there is no permutation $\rho$ of the multiset $\{i,j,k\}$ such that $[\rho(i)\rho(j)\rho(k)]$ appears above.

As shown in~\cite[Section~4]{MGAP17} (cf.~Lemma~\ref{lem_Wallach_subal} above), the proper Lie subalgebras of $\mathfrak g$ containing $\mathfrak h$ properly are
\begin{align*}
\mathfrak s_1=\mathfrak m_2\oplus\mathfrak h,\qquad \mathfrak s_2=\mathfrak m_3\oplus\mathfrak h.
\end{align*}
Invoking Proposition~\ref{prop_irred} as in the proof of Theorem~\ref{thm_Wallach} leads to the conclusion that $\mathfrak s_1$ and $\mathfrak s_2$ meet the first two requirements of Definition~\ref{def_T-apical}. Because $\Ad(H)|_{\mathfrak m_1}$, $\Ad(H)|_{\mathfrak m_2}$ and $\Ad(H)|_{\mathfrak m_3}$ are pairwise inequivalent, the metric $T$ can be expressed as
\begin{align}\label{T_gen_fl_G2}
-z_1\pi_{\mathfrak m_1}^*B-z_2\pi_{\mathfrak m_2}^*B-z_3\pi_{\mathfrak m_3}^*B
\end{align}
with $z_1,z_2,z_3>0$. Formula~\eqref{sigma_str_irred} implies that 
\begin{align*}
\sigma(\mathfrak s_1,T)=\frac1{4z_2}(2-[112]-2[123])=\frac1{12z_2},\qquad
\sigma(\mathfrak s_2,T)=\frac1{4z_3}(2-[123])=\frac3{8z_3}.
\end{align*}
If $z_2/z_3\le2/9$, then $\mathfrak s_1$ meets requirement~3 of Definition~\ref{def_T-apical}. In this case, $\mathfrak s_1$ is a $T$-apical subalgebra of~$\mathfrak g$. According to Theorem~\ref{thm_gen} and formula~\eqref{cond_str_const}, a Riemannian metric with Ricci curvature $cT$ exists for some $c>0$ provided that
\begin{align}\label{lasttag1}
z_1+z_3<12z_2.
\end{align}
If $z_2/z_3\ge2/9$, then $\mathfrak s_2$ meets requirement~3 of Definition~\ref{def_T-apical}. In this case, $\mathfrak s_2$ is $T$-apical. By Theorem~\ref{thm_gen} and~\eqref{cond_str_const}, a metric with Ricci curvature $cT$ exists for some $c>0$ as long as
\begin{align}\label{lasttag2}
6z_1+3z_2<10z_3.
\end{align}
\end{example}

\begin{remark}\label{rem_compar1}
Let $M$ be the generalised flag manifold $G_2/U(2)$ as in Example~\ref{example_G2U2}. Assume the decomposition~\eqref{m_decomp} is as in~\cite{SAIC11}. If $T$ is given by~\eqref{T_gen_fl_G2}, the main result of~\cite{MGAP17} implies that a metric with Ricci curvature $cT$ exists for some $c>0$ provided that
\begin{align}\label{cond_MGAP17}
\frac{z_2}{z_1+z_3}>\frac1{12},\qquad \frac{z_3}{2z_1+z_2}>\frac3{10};
\end{align}
see~\cite[Example~4.3]{MGAP17}. This condition is more restrictive than the one given by Theorem~\ref{thm_gen}. Indeed, it is obvious that both~\eqref{lasttag1} and~\eqref{lasttag2} follow from~\eqref{cond_MGAP17}. On the other hand, suppose, for instance, that $z_2=2z_3/9$. A metric with Ricci curvature $cT$ exists for some $c>0$ as long as
$z_1<5z_3/3$, according to the arguments in Example~\ref{example_G2U2} above. However, inequalities~\eqref{cond_MGAP17} only hold if $z_1<14z_3/9$.
\end{remark}

\begin{example}\label{example_F4}
Suppose $M$ is the generalised flag manifold $F_4/SU(3)\times SU(2)\times U(1)$. Let the decomposition~\eqref{m_decomp} be as in~\cite{AAIC10}. Then
\begin{align*}
s=4,\qquad &d_1=12,\qquad d_2=18,\qquad d_3=4,\qquad d_4=6, \notag \\ &[224]=2,\qquad [112]=2,\qquad [123]=1,\qquad [134]=2/3.
\end{align*}
The structure constant $[ijk]$ vanishes if there is no permutation $\rho$ of the multiset $\{i,j,k\}$ such that $[\rho(i)\rho(j)\rho(k)]$ appears above.

Arguing as in the proof of Lemma~\ref{lem_Wallach_subal}, we can show that the proper Lie subalgebras of $\mathfrak g$ containing $\mathfrak h$ properly are
\begin{align*}
\mathfrak t_1=\mathfrak m_3\oplus\mathfrak h,\qquad \mathfrak t_2=\mathfrak m_4\oplus\mathfrak h,\qquad \mathfrak t_3=\mathfrak m_2\oplus\mathfrak m_4\oplus\mathfrak h.
\end{align*}
In order to apply Theorem~\ref{thm_gen}, we need to determine which of these subalgebras are $T$-apical. Obviously, they all satisfy requirement~1 of Definition~\ref{def_T-apical}. Moreover, $\mathfrak t_1$ and $\mathfrak t_3$ are maximal in $\mathfrak g$, while $\mathfrak t_2$ is not. Proposition~\ref{prop_irred} implies that $\mathfrak t_1$ and $\mathfrak t_2$ satisfy requirement~2 of Definition~\ref{def_T-apical}. The equality
\begin{align}\label{T_gen_fl_F4}
T=-z_1\pi_{\mathfrak m_1}^*B-z_2\pi_{\mathfrak m_2}^*B-z_3\pi_{\mathfrak m_3}^*B-z_4\pi_{\mathfrak m_4}^*B
\end{align}
holds for some $z_1,z_2,z_3,z_4>0$. 
Exploiting~\eqref{sigma_str_irred}, we find
\begin{align*}
\sigma(\mathfrak t_1,T)=\frac1{4z_3}(2-[123]-[134])=\frac1{12z_3},\\
\sigma(\mathfrak t_2,T)=\frac1{12z_4}(6-2[413]-[422])=\frac2{9z_4}.
\end{align*}
Our next step is to understand whether $\mathfrak t_3$ meets requirement~2 of Definition~\ref{def_T-apical} and to compute $\sigma(\mathfrak t_3,T)$.

Given $h\in\mathcal M(\mathfrak t_3)$, the equality
\begin{align*}
h=-u\pi_{\mathfrak m_2}^*B-v\pi_{\mathfrak m_4}^*B
\end{align*}
holds for some $u,v>0$. By~\eqref{hat_S_formula},
\begin{align*}
\hat S(h)=\frac{18}{2u}+\frac6{2v}-\frac{[211]}{2u}-\frac{[213]}u-\frac{[413]}v-\frac{[224]v}{4u^2}-\frac{[224]}{2v}
=\frac7u+\frac4{3v}-\frac v{2u^2}.
\end{align*}
If $h\in\mathcal M_T(\mathfrak t_3)$, then
\begin{align}\label{lasttag3}
\tr_hT|_{\mathfrak t_3\ominus\mathfrak h}=\frac{18z_2}u+\frac{6z_4}v=1,\qquad u=\frac{18vz_2}{v-6z_4},\qquad v>6z_4.
\end{align}
In this case, $\hat S(h)=\psi_{z_2,z_4}(v)$ with the function $\psi_{z_2,z_4}:(6z_4,\infty)\to\mathbb R$ defined by
\begin{align*}
\psi_{z_2,z_4}(x)=\frac{7(x-6z_4)}{18xz_2}+\frac4{3x}-\frac{(x-6z_4)^2}{648xz_2^2}.
\end{align*}
Consequently,
\begin{align}\label{lasttag4}
\sigma(\mathfrak t_3,T)=\sup\{\hat S(h)\,|\,h\in\mathcal M_T(\mathfrak t_3)\}=\sup\{\psi_{z_2,z_4}(v)\,|\,v>6z_4\}.
\end{align}
A straightforward computation shows that
\begin{align*}
\frac d{dv}\psi_{z_2,z_4}(v)=\frac1{3v^2}\Big(\frac{z_4^2}{6z_2^2}+\frac{7z_4}{z_2}-4\Big)-\frac1{648z_2^2}.
\end{align*}
If $z_2/z_4\ge7/4$, then $\frac d{dv}\psi_{z_2,z_4}(v)<0$ for all $v>6z_4$. In this case, since the function $\psi_{z_2,z_4}$ does not have a global maximum on $(6z_4,\infty)$, the supremum of $\hat S$ over $\mathcal M_T(\mathfrak t_3)$ is not attained. As a consequence, $\mathfrak t_3$~fails to satisfy requirement~2 of Definition~\ref{def_T-apical}. Also,
\begin{align*}
\sigma(\mathfrak t_3,T)=\lim_{v\to6z_4}\psi_{z_2,z_4}(v)=\frac2{9z_4}.
\end{align*}
If $z_2/z_4<7/4$, then $\psi_{z_2,z_4}$ attains its global maximum at the point
\begin{align*}
v_0=6\sqrt{z_4^2+42z_2z_4-24z_2^2}.
\end{align*}
Define $h_0\in\mathcal M(\mathfrak t_3)$ by the formula
\begin{align*}
h_0=-\frac{18v_0z_2}{v_0-6z_4}\pi_{\mathfrak m_2}^*B-v_0\pi_{\mathfrak m_4}^*B.
\end{align*}
Recalling~\eqref{lasttag3} and~\eqref{lasttag4}, we conclude that $h_0$ lies in $\mathcal M_T(\mathfrak t_3)$ and 
\begin{align*}
\hat S(h_0)=\psi_{z_1,z_2}(v_0)=\sigma(\mathfrak t_3,T).
\end{align*}
Thus, $\mathfrak t_3$ satisfies requirement~2 of Definition~\ref{def_T-apical}. Also,
\begin{align*}
\sigma(\mathfrak t_3,T)=\psi_{z_2,z_4}(v_0)=\psi(z_2,z_4)
\end{align*}
with the function $\psi:\{(x,y)\in(0,\infty)^2\,|\,x<7y/4\}\to\mathbb R$ given by
\begin{align*}
\psi(x,y)=
\frac7{18x}-\frac{\sqrt{y^2+42xy-24x^2}-y}{54x^2}.
\end{align*}

We are now ready to apply the results of Section~\ref{sec_gen_results}.
If the inequalities
\begin{align}\label{cond01F4}
z_2/z_4\ge7/4,\qquad z_3/z_4\le3/8,
\end{align}
or the inequalities
\begin{align}\label{cond11F4}
z_2/z_4<7/4,\qquad 1/z_3\ge12\psi(z_2,z_4),
\end{align}
hold, then
$\sigma(\mathfrak t_1,T)\ge\sigma(\mathfrak t_3,T).$
In this case, $\mathfrak t_1$ meets requirement~3 of Definition~\ref{def_T-apical}, which means $\mathfrak t_1$ is $T$-apical. According to Theorem~\ref{thm_gen} and formula~\eqref{cond_str_const}, a Riemannian metric with Ricci curvature $cT$ exists for some $c>0$ provided that
\begin{align}\label{suff_F4_1}
2z_1+3z_2+z_4<30z_3.
\end{align}
If
\begin{align}\label{cond02F4}
z_2/z_4\ge7/4,\qquad z_3/z_4\ge3/8,
\end{align}
then
\begin{align*}
\sigma(\mathfrak t_2,T)=\sigma(\mathfrak t_3,T)\ge\sigma(\mathfrak t_1,T).
\end{align*}
In this case, $\mathfrak t_2$ is $T$-apical. By Theorem~\ref{thm_gen} and formula~\eqref{cond_str_const}, a metric with Ricci curvature $cT$ exists for some $c>0$ provided that
\begin{align}\label{suff_F4_2}
12z_1+18z_2+4z_3<63z_4.
\end{align}
Finally, if
\begin{align}\label{cond22F4}
z_2/z_4<7/4,\qquad 1/z_3\le12\psi(z_2,z_4),
\end{align}
then $\sigma(\mathfrak t_3,T)\ge\sigma(\mathfrak t_1,T)$. In this situation, $\mathfrak t_3$ is $T$-apical. A metric with Ricci curvature $cT$ exists for some $c>0$ as long as
\begin{align}\label{suff_F4_3}
(3z_1+z_3)\psi(z_2,z_4)<2.
\end{align}
\end{example}

\begin{remark}
In Example~\ref{example_F4}, if $z_2/z_4\ge7/4$, one may use Lemma~\ref{lem_key_apic_ex} rather than a direct computation to find $\sigma(\mathfrak t_3,T)$. Indeed, the generalised flag manifold $F_4/SU(3)\times SU(2)\times U(1)$ satisfies Hypothesis~\ref{hyp_strong}. According to Lemma~\ref{lem_key_apic_ex}, there exists a Lie subalgebra $\mathfrak k$ of $\mathfrak t_3$ such that formulas~\eqref{apic_ex_flas} hold with some $h\in\mathcal M_T(\mathfrak k)$ and $\mathfrak s=\mathfrak t_3$. If $z_2/z_4\ge7/4$, then the supremum of $\hat S$ over $\mathcal M_T(\mathfrak t_3)$ is not attained and $\mathfrak k$ cannot equal~$\mathfrak t_3$. In this case, $\mathfrak k=\mathfrak t_2$ and $\sigma(\mathfrak t_2,T)\ge\sigma(\mathfrak t_3,T)$. On the other hand, Remark~\ref{rem_sigma_subalg} implies $\sigma(\mathfrak t_2,T)\le\sigma(\mathfrak t_3,T)$. Thus,
\begin{align*}
\sigma(\mathfrak t_2,T)=\sigma(\mathfrak t_3,T)=\frac2{9z_4}.
\end{align*}
\end{remark}

\begin{remark}\label{rem_compar2}
Let $M$ be the generalised flag manifold $F_4/SU(3)\times SU(2)\times U(1)$ as in Example~\ref{example_F4}. Assume the decomposition~\eqref{m_decomp} is as in~\cite{AAIC10}. If $T$ is given by~\eqref{T_gen_fl_F4}, the main result of~\cite{MGAP17} implies that a metric with Ricci curvature $cT$ exists for some $c>0$ provided that
\begin{align}\label{cond_F4_MGAP}
\frac{z_4}{z_2}>\frac47,\qquad \frac{z_3}{2z_1+3z_2+z_4}>\frac1{30}, \qquad 
\frac{\min\{z_2,z_4\}}{3z_1+z_3}>\frac{47}{72}.
\end{align}
This condition is more restrictive than the one given by Theorem~\ref{thm_gen}. Indeed, suppose $z_2/z_4\ge7/4$. According to Theorem~\ref{thm_gen}, a metric with Ricci curvature $cT$ exists for some $c>0$ as long as~\eqref{cond01F4} and~\eqref{suff_F4_1} or~\eqref{cond02F4} and~\eqref{suff_F4_2} are satisfied, as explained in Example~\ref{example_F4}. However, inequalities~\eqref{cond_F4_MGAP} fail to hold. Next, suppose $z_2/z_4<7/4$. By Theorem~\ref{thm_gen}, a metric with Ricci curvature $cT$ exists for some $c>0$ as long as~\eqref{cond11F4} and~\eqref{suff_F4_1} or~\eqref{cond22F4} and~\eqref{suff_F4_3} are satisfied, as shown in Example~\ref{example_F4}. One can derive both~\eqref{suff_F4_1} and~\eqref{suff_F4_3} from~\eqref{cond_F4_MGAP}.
At the same time,~\eqref{cond22F4} and~\eqref{suff_F4_3} hold if, for instance,
\begin{align*}
z_2=z_3=z_4,\qquad z_1\in\bigg[\frac{25z_2}{141},\frac{(637+36\sqrt{19})z_2}{465}\bigg).
\end{align*}
However, such $z_1$, $z_2$, $z_3$ and $z_4$ do not satisfy~\eqref{cond_F4_MGAP}.
\end{remark}

\section*{Acknowledgement}

I am grateful to Andreas Arvanitoyeorgos and Marina Statha for suggesting that I consider the prescribed Ricci curvature problem on generalised Wallach spaces.

\end{document}